\documentclass[12pt]{article}
\usepackage{amsmath,amsthm,amsfonts,amssymb,amscd}
\usepackage[latin1]{inputenc}
\usepackage{xcolor}
\usepackage[mathcal]{eucal}
\usepackage{psfrag}
\usepackage{graphicx}
\usepackage{epstopdf}
\theoremstyle{plain}
\newtheorem{lemma}{Lemma}[section]
\newtheorem{corollary}[lemma]{Corollary}

\newtheorem{maintheorem}{Theorem}

\newcommand{\interior}{\operatorname{int}}
\newcommand{\osc}{\operatorname{osc}}

\newcommand{\dist}{\operatorname{dist}}
\newcommand{\graph}{\operatorname{graph}}
\theoremstyle{definition}

\newcommand{\halpha}{\alpha^{\frac{1}{2(D-1)}}}
\newcommand{\Dalpha}{\sqrt[D]{\alpha}}

\newcommand{\abs}[1]{\left\lvert#1\right\rvert}

\title{Building Expansion for Generalizations of Viana Maps}
\author{V. Horita, N. Muniz, and O. Sester
\thanks{Work partially supporte by FAPESP (2014/21815-5 and 2019/10269-3).}}
\date{\today}

\begin{document}

\maketitle

\let\thefootnote\relax
\let\thefootnote\relax

\begin{abstract}
In the seminal paper \cite{Vi97}, Viana built examples of maps presenting 
two positive Lyapunov exponents 
exploring  skew-products of a (uniformly) expanding map and a quadratic map (order 2 critical point) 
perturbed by some level of \emph{noise}.
Here we extend that construction replacing the quadratic underlying dynamics 
by maps with a more degenerated critical point.  
\end{abstract}

%============================================================================

\section{Introduction}
\label{s.intr}

In dimension one, much has been studied about the richness of chaotic behavior (positive Lyapunov exponent) present in families at a parameter, see Jakobson \cite{Ja81}, Benedicks and Carleson \cite{BC85}, Graczyk and Swiatek \cite{GS97}, and Koslovski, Shen, and van Strien \cite{KSS07}. 
On the other hand, this phenomenon does not occur persistently since hyperbolicity is open and dense for one dimensional $C^k$-maps.
  
In higher dimension, Viana \cite{Vi97} brought us an impressive account of the existence of multidimensional nonhyperbolic behaviour in a persistent way.
His work introduced what now is referred as Viana maps, i.e., $C^3$-per\-tur\-ba\-tion of a $C^3$-skew-product 
$\varphi_\alpha \colon \mathbb{S}^1 \times \mathbb{R} \to \mathbb{S}^1 \times \mathbb{R}$ 
given by $\varphi_\alpha(\theta, x) = (\hat{g}(\theta) , a(\theta) - x^2)$,
where $\hat{g} \colon \mathbb{S}^1 \to \mathbb{S}^1$ is an expanding map of the circle and $a(\theta) = a_0 + \alpha \phi(\theta)$.
The map $\phi$ is a Morse function and $a_0 \in (1,2)$ is fixed such that $x=0$ is a pre-periodic point for the map $h(x) = a_0 - x^2$.
There exists a compact interval $I_0 \subset (-2,2)$ such that $\varphi_\alpha(\mathbb{S}^1 \times I_0) \subset \interior (\mathbb{S}^1 \times I_0)$.
Particular cases that satisfy the above condition are $\phi(\theta) = \sin 2\pi \theta$ and and $\hat{g}(\theta) = d \theta \mod 1$.
Thus, $\varphi_\alpha$ has the form.
$$
\varphi_\alpha(\theta, x) = (d \theta \mod 1 , \alpha \sin 2\pi \theta + h(x)).
$$
Viana showed that $C^3$-perturbation of  $\varphi_\alpha$ has two positive Lyapunov exponents, for $d \ge 16$.
Later, this result was extended for $d \ge 2$ by Buzzi, Sester, and Tsuji in \cite{BST03}.
The factor $\alpha \sin 2\pi \theta$ can be though as an $\alpha$-perturbation of the quadratic map $h$.

Skew products of this type, with a curve of neutral fixed points were exploited by Gouzel \cite{G06}.
More extensions of the result were obtained by Schnellmann in \cite{Sch08} and \cite{Sch09} with  Misiurewicz-Thurston quadratic maps 
as the basis dynamics, Huang-Shen in \cite{HS13} and Varandas \cite{Varandas12} in a generalized context.
Moreover, Gao in \cite{Gao21} consider Benedicks-Carleson maps instead circle uniformly maps $\hat{g}$ in the first factor to obtain two positive Lyapunov exponents.
All works mentioned deal with a quadratic critical point.

Here, we consider maps $h_D \colon \mathcal{M} \to \mathcal{M}$, $D \ge 2$ (to be defined in a little while), where $\mathcal{M}=\mathbb{S}^1$ if $D$ is odd and $\mathcal{M} = I_0 \subset \mathbb{R}$ an interval if $D$ is even.
The map $h_D$ presents one critical point $\tilde{x}$ of order $D$, that is, $h_D^{(j)}(\tilde{x}) = 0$, for all $1 \le j < D$, and $h_D^{(D)}(\tilde{x}) \neq 0$. 
Let $\varphi_{\alpha,D} \colon \mathbb{S}^1 \times \mathcal{M} \to \mathbb{S}^1 \times \mathcal{M}$ be defined by
$$
\varphi_{\alpha,D}(\theta, x) = (d \theta \mod 1 , \alpha \sin 2\pi \theta + h_D(x)).
$$ 

Let us be more specific on $h_D$.
We consider intervals $I',I'' \subset \mathcal{M}$ centered in $\tilde{x}$ depending on $D$, such that $I''$ is a proper sub-interval of $I'$. 
We begin defining the map for positive odd number $D \ge 3$.
In \cite{HMS07} are considered one-parameter families maps in $\mathbb{S}^1$ with one cubic critical point and, for a positive measure set of parameter, the maps are non-uniformly expanding.
Here, we consider $C^{2D+2}$-maps $h_{2D+1} \colon \mathbb{S}^1 \to \mathbb{S}^1$, $D \ge 1$ positive integer, defined by 
\begin{equation}
\label{e.hD}
h_{2D+1}(x) = \left\{
\begin{array}{lll}
2x \mod 1 & \text{if} & x\in \mathbb{S}^1 \setminus I' \\
A(x-1/2)^{2D+1} & \text{if} & x \in I'' 
\end{array} 
\right. ,
\end{equation}
where $A$ is a positive constant chosen such that the derivative of $h_{2D+1}$ at the extreme points of $I''$ is equal to $7/4$. 
In each component of $I' \setminus I''$ we suppose that the first and second derivatives of $h_{2D+1}$ are monotone.
Then, $\tilde{x} = 1/2$ is the unique critical point of $h_{2D+1}$.
\begin{figure}[phtb]
\psfrag{0}{$0$}
\psfrag{1}{$1/2$}
\psfrag{-1}{$-1/2$}
\psfrag{2}{$1$}
\centering
\includegraphics[height=6cm]{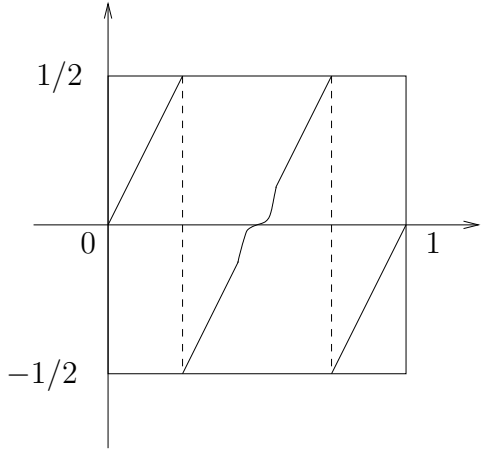}
\caption{Map with odd critical point} 
\label{f.example}
\end{figure}

Now, we take $I' = (-1,1)$ and we consider $h_{2D} \colon I_0 \to I_0$, $D \ge 1$ positive integer, a $C^{2D+1}$-map defined by  
\begin{equation}
\label{e.h2D}
h_{2D}(x) = \left\{
\begin{array}{lll}
a_0 - x^2 & \text{ if } & x \in I_0 \setminus I'  \\
a_0 - A x^{2D} & \text{ if } & x \in I'' 
\end{array}
\right. ,
\end{equation}
where, as before, $A$ is such that at the extreme points of $I''$ the modulus of derivative of $h_D$ is equal $7/4$, and in each component of $I' \setminus I''$ the first and second derivatives are monotone. 
As defined, $h_{2D}$ has a unique (pre-fixed) critical point $\tilde{x}=0$ of order $2D$.

\begin{maintheorem}
\label{mt.general}
For every sufficiently small $\alpha > 0$, and for all integer number $d \ge 16$, the map $\varphi_{\alpha,D}$, has two positive Lyapunov exponents at Lebesgue every point $(\theta, x) \in \mathbb{S}^1 \times \mathcal{M}$. 
Moreover, the same holds for every map $\varphi$ sufficiently close to $\varphi_{\alpha,D}$ in $C^{D}(\mathbb{S}^1 \times \mathcal{M})$.
\end{maintheorem} 

The proof of Theorem~\ref{mt.general} follows closely the proof in \cite{Vi97} where it is assumed $d \ge 16$. 
The general case $d \ge 2$ can be adapted from \cite{BST03}. 
In this paper, we deal only with the case where $d \ge 16$.

Let us mention that the presence of two positive Lyapunov exponents is an indication of the existence of absolutely continuous invariant probability measure.
For Viana maps and extensions in the case of a quadratic critical point, see Alves \cite{Al00}, Alves-Viana \cite{AV02}, and Gao \cite{Gao21}.

\section{Preliminary statements and results}

We can assume that $\varphi \colon \mathbb{S}^1 \times \mathcal{M} \to \mathbb{S}^1 \times \mathcal{M}$ has the form
\begin{equation}
\label{simplificacao}
\varphi (\theta,x) = (g(\theta), f(\theta,x)), \quad \text{with } \partial_x f(\theta,x) = 0 \text{ if and only if } x = \tilde{x},
\end{equation}
and we prove that the theorem holds as long as 
$$
|\varphi - \varphi_{\alpha,D} |_{C^D} \le \alpha \quad \text{ on } \mathbb{S}^1 \times \mathcal{M}.
$$
To overcome the simplifying hypothesis (\ref{simplificacao}) and so conclude the proof of Theorem \ref{mt.general} it is enough to follow exactly the same approach as in  \cite[Section 2.5]{Vi97}.

\subsection{Admissible curves}
We say that a $C^2$-curve $\hat{X}$ is an \emph{admissible curve} if $\hat{X} = \graph{X}$, with  $X \colon \mathbb{S}^1 \to \mathcal{M}$, satisfies
$$
|X'(\theta)| \le \alpha \quad \text{ and } \quad |X''(\theta)| \le \alpha, \qquad
\text{for all } \theta \in \mathbb{S}^1.
$$
We suppose that $\alpha$ is small such that $\alpha < A^{-1}$.

Given $\omega \subset \mathbb{S}^1$, we write $\hat{X}|_\omega =
\graph (X|_\omega)$.

Let $\tilde{\theta}_0 \in \mathbb{S}^1$ be the fixed point of $g$ and denote by $\tilde{\theta}_1, \cdots , \tilde{\theta}_d = \tilde{\theta}_0$ its $d$ pre-images under $g$ ordered according to the orientation of $\mathbb{S}^1$.
We consider \emph{Markov partitions} $\mathcal{P}_n$ of $\mathbb{S}^1$ defined by,
\begin{align*}
\mathcal{P}_1 & = \{[\tilde{\theta}_{j-1}, \tilde{\theta}_j) \colon 1 \le j \le d \}, \text{ and} \\
\mathcal{P}_{n+1}&  = \{ \text{connected components of } g^{-1}(\omega) \colon \omega \in \mathcal{P}_n \}.
\end{align*}

\begin{lemma}
\label{v.lemma2.1}
Let $X$ be an admissible curve. If $\omega \in \mathcal{P}_n$ then $\varphi^n(\hat{X}|_\omega)$ is
also an admissible curve.
\end{lemma}

\begin{proof}
The proof follows closely \cite[Lemma 2.1]{Vi97}. Let $Y \colon \mathbb{S}^1 \to I$ be defined by
$Y(g(\theta)) = f(\theta, X(\theta))$, $\theta \in \omega \in \mathcal{P}_1$.
The estimates on the derivatives of $g$ are taken from \cite{Vi97}, i.e., $|g'| \ge 15$ and
$|g''| \le \alpha$.

The derivatives of $f$ are $\alpha$ perturbations of those derivatives of $\tilde{f}(\theta,x) = \alpha \sin 2\pi \theta + h_D(x)$. 
So,
\begin{align*}
& |\partial_\theta f| \le \alpha |2\pi \cos 2\pi \theta| + \alpha \le 8 \alpha, \\
& |\partial_{\theta \theta} f| \le \alpha |4 \pi^2 \sin 2\pi \theta| + \alpha \le 50 \alpha, \text{ and } \\
&  |\partial_{\theta x} f| \le \alpha.
\end{align*}

Note that we have $|h_D'(x)| \le 7/4$ in $I''$ and $7/4 \le |h_D'(x)| \le 2$ in $I' \setminus I''$. 
In $\mathcal{M} \setminus I'$, $|h_D'(x)| = 2$ if $D$ is odd, and $|h_D'(x)| \le 4$ if $D$ is even.
Anyway, $|h_D'(x)| \le 4$ for every $x \in \mathcal{M}$. 

For the second derivative, in $\mathcal{M} \setminus I'$, we have $|h_D''(x)| = 0$ if $D$ is odd, and $|h_D''(x)| \le 2$ if $D$ is even.
We are assuming that in the extreme points $x_i$, $i = 1,2$, of $I''$ the derivative has modulus $7/4$, that is to say 
\begin{equation}
7/4 = |h_D'(x_i)| = DA|x_i-\tilde{x}|^{D-1} = D A \left(\frac{|I''|}{2}\right)^{D-1} .
\end{equation}
Then,
$$
|h_D''(x)| \le D (D-1) A |x - \tilde{x}|^{D-2} \le D (D-1) A \left(\frac{|I''|}{2}\right)^{D-2} \le \frac{7(D-1)}{4|I''|}.
$$

From the fact that $|h_{2D+1}'' (x)| = 0$ for every $x \in \mathcal{M} \setminus I'$, $|h_{2D}''(x)| = 2$ for every $x \in \mathcal{M} \setminus I'=(-1,1)$, and $h_D''$ is monotone in each connected component of $I' \setminus I''$, we have $|h_D''(x)| < 7(D-1)/(4|I''|)$, for every $x$ in $\mathcal{M}$.

The fact that $\partial_x \tilde{f}(\theta,x) = h_D'(x)$ and $\partial_{xx} \tilde{f}(\theta,x) = h_D''(x)$ implies, for all $x \in \mathcal{M}$,
$$
|\partial_x f| \le 4 + \alpha \qquad \text{ and } \quad |\partial_{xx} f| \le \frac{7(D-1)}{4|I''|} + \alpha \le 2 \alpha^{-1},
$$
where the last inequality uses that $\alpha$ is small, and so we may assume $\alpha < \left(7(D-1)/(4|I''|)\right)^{-1}$.
Thus,
\begin{align*}
|Y'| & = \left| \frac{1}{g'} (\partial_\theta f + \partial_x f X') \right| \le
\frac{1}{15}(8\alpha + 5 \alpha) \le \alpha  \quad \text{ and}\\
|Y''| & = \left| \left( \frac{1}{g'} \right)^2 (\partial_{\theta\theta} f +
2 \partial_{\theta x} f X' + \partial_{xx}f(X')^2 + \partial_x f X'' - Y'g'') \right| \le \\
& \le \frac{1}{15^2}(50\alpha + 2 \alpha + \frac{7(D-1)}{4|I''|} \alpha^2 + 5\alpha + \alpha^2) \le \alpha.
\end{align*}
The lemma is proved.
\end{proof}

Next, we state a property of admissible curves.

\begin{lemma}
\label{l.lemma2.2}
Let $\hat{X} = \graph{X}$ be an admissible curve and denote $\hat{X} (\theta) = (\theta, X(\theta))$, $\hat{Z} = \varphi(\hat{X}(\theta)) = (g(\theta),Z(\theta))$.
Then, given any interval $I\subset \mathcal{M}$, we have $m( \{\theta \in \mathbb{S}^1 \colon \hat{Z}(\theta) \in \mathbb{S}^1 \times I \}) \le 4|I|/\alpha + 2 \sqrt{|I|/\alpha}$.
\end{lemma}

\begin{proof}
Analogous to \cite[Lemma 2.2]{Vi97}.
\end{proof}

\begin{corollary}
\label{c.corollary2.3}
there is $C_1>0$ such that, given $\hat{X}_0 = \graph{X_0}$ an admissible curve and $I\subset \mathcal{M}$ an interval with $|I| \le \alpha$ we have
$$
m( \{\theta \in \mathbb{S}^1 \colon \hat{X_j}(\theta) \in \mathbb{S}^1 \times I \}) \le C_1 \sqrt{\frac{|I|}{\alpha}}, \text{ for every } j \ge 1.
$$
\end{corollary}

\begin{proof}
Analogous to \cite[Corollary 2.3]{Vi97}.
\end{proof}

For $(\theta, x) \in \mathbb{S}^1 \times \mathcal{M}$ and $j \ge 0$ we write
$(\theta_j,x_j) = \varphi^j (\theta,x)$.
We consider constants $0< \eta \le 1/3$ and $0< \kappa <1$, that will be made precise in a little while, both depending only on $h_D$.

\begin{lemma}
\label{v.lemma2.4}
There are $\delta_1 > 0$ and $\sigma_1>1$ such that
\begin{itemize}
\item[a)] for every $\alpha > 0$ small, there exists $N = N(\alpha) \ge 1$ such that
$$
\prod_{j=0}^{N-1} \abs{\partial_x f(\theta_j,x_j)} \ge \abs{x-\tilde{x}}^{D-1} \alpha^{-1 + \frac{\eta}{D-1}},
$$
whenever $\abs{x-\tilde{x}} < 2 \Dalpha$;
\item[b)] for each $(\theta,x) \in \mathbb{S}^1 \times \mathcal{M}$ with $\Dalpha \le \abs{x - \tilde{x}} < \delta_1$,
there exists $p(x) \le N$ such that
$$
\prod_{j=0}^{p(x)-1} \abs{\partial_x f(\theta_j,x_j)} \ge \frac{1}{\kappa} \sigma_1^{p(x)}.
$$
\end{itemize}
\end{lemma}

\begin{proof} 
The proof follows that of \cite[Lemma 2.4]{Vi97} and we also denote
$C$ any large constant depending only on the map $h_D$.

If $D$ is odd (respectively, even), we consider an interval $J$ in $\mathbb{S}^1$ (respectively, in $I_0$) centered at $0$ (respectively, at the negative fixed point $q$ of $h_D$).
We have $h(I'') \subset J$ (respectively, $h_D^2(I'') \subset J$), if $I''$ is sufficiently small. 

First, we consider $x \in I''$ with $|x - \tilde{x}| \ge \halpha$, let $k_0 = k_0(y)$ (uniformly bounded) 
be the smallest positive integer $j$ such that $h_D^j(y)$ escape from $J$.
Recall that for $D$ odd we have $\tilde{x} = 1/2$ (respectively, for $D$ even we have $\tilde{x} = 0$).
Note that $h_{D}(1/2) = 0$ (respectively, $h_D^2(0) = q$) 
is the fixed point of $h_{D}$ and $h_{D}'(0) = \rho = 2$ (respectively, $h_D' (q) = \rho \le 4$).
We write $\tilde{q} = 0$ for $D$ odd (respectively, $\tilde{q} = q$ for $D$ even).

We fix $\rho_1 < \rho < \rho_2$ with $\rho_1 > \rho_2^{1-\eta/D}$.
Let $\delta_0 = |J|/2$.
We take $J$ sufficiently small such that  for every $|y - \tilde{q}| < \delta_0$, we
have
$$
\rho_1 < |\partial_x f(\varphi (\tau, y))| < \rho_2.
$$

Given $(\theta,x) \in \mathbb{S}^1 \times \mathcal{M}$ we denote $d_i =
|x_{\ell+i} - \tilde{q}|$, $i \ge 0$, $\ell = 1$ if $D$ is odd (respectively, $\ell = 2$ if $D$ is even).
We suppose $\delta_1 =  |I''|/2 > 0$ and $\alpha$ small enough so that $|x - \tilde{x}|< \delta_1$
implies
$$
d_0 \le C |x -\tilde{x}|^D + C\alpha < \delta_0.
$$
Now, let $(\theta,x)$ and $i \ge 1$ be such that $|x - \tilde{x}|< \delta_1$ and
$d_0, \dots , d_{i-1} < \delta_0$.
Thus, $d_i \le \rho_2 d_{i-1} + C \alpha$ and by induction
\begin{equation}
\label{eq.5}
d_i \le \rho_2^i d_0 + C \alpha (1 + \rho_2+ \dots + \rho_2^{i-1}) \le
\rho_2^i (C \alpha + C |x - \tilde{x}|^D).
\end{equation}
Suppose first that $|x-\tilde{x}| < 2\Dalpha$. 
Then,  $|x-\tilde{x}|^D <  2^D\alpha$ and from \eqref{eq.5}, we get $\displaystyle d_i \le \rho_2^i C \alpha$.

Let $\tilde{N} = \tilde{N}(\alpha) \ge 1$ be the minimum integer such that $\rho_2^{\tilde{N}}
C \alpha \ge \delta_0$ and then define $N= \ell +\tilde{N}$.
Hence, $d_i = |x_{\ell+i} - \tilde{q}| < \delta_0$ for all $0 \le i \le N-1$ and
\begin{align*}
\prod_{j=0}^{N-1} |\partial_x f(\theta_j,x_j)| & = \prod_{j=0}^{\ell-1} |\partial_x f(\theta_j,x_j)|
\cdot \prod_{j=0}^{\tilde{N}-1} |\partial_x f(\theta_{\ell + j},x_{\ell + j}| \ge \frac{1}{C}
|x - \tilde{x}|^{D-1} \rho_1^{\tilde{N}} \\
& \ge \frac{1}{C} |x - \tilde{x}|^{D-1} \rho_2^{(1 - \eta/D) \tilde{N}} \ge
\frac{1}{C} |x - \tilde{x}|^{D-1} \alpha^{-1 + \eta/D} \\
& \ge |x - \tilde{x}|^{D-1} \alpha^{-1 + \eta/(D-1)}.
\end{align*}
Part a) is proved.

Suppose now $|x-\tilde{x}| \ge \Dalpha$. 
Then \eqref{eq.5} gives $d_i \le \rho_2^i C |x-\tilde{x}|^D $.
Let $\tilde{p}(x)$ be the minimum integer such that $\rho_2^{\tilde{p}(x)} C |x - \tilde{x}|^D
\ge \delta_0$.
Define $p(x) = \ell + \tilde{p}(x)$.
Then, as before,
\begin{align*}
\prod_{j=0}^{p(x)-1} |\partial_x f(\theta_j,x_j)| & \ge  \frac{1}{C}
|x - \tilde{x}|^{D-1} \rho_1^{\tilde{p}(x)} \ge \frac{1}{C}
\left( \frac{\rho_1}{\rho_2^{(D-1)/D}}\right)^{\tilde{p}(x)} \\
& \ge \frac{1}{C} \left(
\frac{\rho_2^{1 - \eta/D}}{\rho_2^{(D-1)/D}}\right)^{\tilde{p}(x)}
\ge \frac{1}{C} \rho_2^{(1/D - \eta/D)\tilde{p}(x)} \\
& \ge \frac{1}{\kappa} \rho_2^{p(x)/(D+1)}.
\end{align*}
The last inequality follows from the fact that $p(x) \gg 1$ uniformly as
long as $\delta_1 \ll \delta_0$.
So, if we take $\sigma_1 = \rho_2^{1/(D+1)}$ the result follows.
The proof is complete.
\end{proof}

\begin{lemma}
\label{l.lemma2.5}
There exist $\sigma_2 > 1$ and $C_2> 0$ such that
$$
\prod_{j=0}^{k-1} \abs{\partial_x f(\theta_j,x_j)} \ge C_2 \sqrt[D]{\alpha^{D-1}}  \; \sigma_2^k,
$$
for every $(\theta,x)$ with $\abs{x_0-\tilde{x}}, \abs{x_1-\tilde{x}} , \dots ,
\abs{x_{k-1}-\tilde{x}} \ge \Dalpha$.
If in addition $\abs{x_k-\tilde{x}} < \delta_1$, then
$$
\prod_{j=0}^{k-1} \abs{\partial_x f(\theta_j,x_j)} \ge C_2 \sigma_2^k.
$$
\end{lemma}

\begin{proof}
We consider $D \ge 2$ and keep the notations of previous lemma. 

Let us recall that $\delta_1 = |I''|/2$. 
From the hypothesis on $h_D$,  if $\sigma_0 = 7/4$ then
$$
|(h_D^{n})'(y)| \ge \sigma_0^{n},
$$
whenever $|y - \tilde{x}| , \dots , |y_{n-1} - \tilde{x}| \ge \delta_1$.
By continuity, supposing $\alpha$ small enough and reducing $\sigma_0$ if necessary, if $(\tau, y) \in
\mathbb{S}^1 \times \mathcal{M}$ with
$|y_0 -\tilde{x}|, \cdots , |y_{n-1} - \tilde{x}| \ge \delta_1$ then
\begin{equation}
\label{v.eq6}
\prod_{j=0}^{n-1} |\partial_x f (\tau_j, y_j)| \ge \sigma_0^{n}.
\end{equation}

Let $(\theta, x)$ be as in the statement and let $j_1 < \cdots < j_s$ be
the values of $j \in \{0, \cdots, k-1\}$ for which
$|x_j - \tilde{x}| < \delta_1$.
If $s = 0$ the result follows immediately from \eqref{v.eq6}.
So, let us suppose $s >0$.
When $|x_k - \tilde{x}| < \delta_1$ we set $j_{s+1} = k$.
Denoting $p_i = p(x_{j_i})$, $i=1, \cdots, s$, by Lemma~\ref{v.lemma2.4}
we have
\begin{equation}
\label{v.eq9}
\prod_{j=j_i}^{j_i + p_i-1} |\partial_x f (\theta_j, x_j)| \ge
\frac{1}{\kappa} \sigma_1^{p_i},
\end{equation}
for all $i<s$.
Moreover, if $j_s + p_s \le k$ then \eqref{v.eq9} also holds for $i=s$,
this is 
the case if $|x_k - \tilde{x}| < \delta_1$, as the definition of $p(x)$
implies $j_i + p_i < j_{i+1}$.
It follows from \eqref{v.eq6} that
\begin{equation}
\label{v.eq10}
\prod_{j=0}^{j_i} |\partial_x f (\theta_j, x_j)| \ge \sigma_0^{j_1} \quad
\text{ and }
\prod_{j=j_i + p_i}^{j_{i+1}-1} |\partial_x f (\theta_j, x_j)| \ge
\sigma_0^{j_{i+1} - j_i - p_i},
\end{equation}
for all $i<s$ and, again, the second inequality remains valid for $i = s$
when $|x_k - \tilde{x}| < \delta_1$.
Now, we take $\sigma_2 = \min \{ \sigma_0, \sigma_1 \}$ and we get
$$
\prod_{j=0}^{k-1} |\partial_x f (\theta_j, x_j)| \ge \sigma_0^{j_1}
\prod_{i=1}^s \sigma_1^{p_i} \sigma_0^{j_{i+1} - j_i - p_i}\ge \sigma_2^k,
$$
whenever $|x_k - \tilde{x}| < \delta_1$. The second part of the lemma is proved.

For the first part, note that, even if \eqref{v.eq9} and \eqref{v.eq10} are not valid 
for $i=s$,  we have
$$
\prod_{j=j_s}^{k-1} |\partial_x f (\theta_j, x_j)| \ge (DA - \alpha)
|x_{j_s} - \tilde{x}|^{D-1}  \sigma_0^{k-j_s -1} \ge 
C_2 \sqrt[D]{\alpha^{D-1}} \sigma_0^{k-j_s -1},
$$
as a consequence of \eqref{v.eq6}.
The lemma is proved.
\end{proof}

Consider now $\eta = \log \sigma_2/(4\log 32) \le 1/3$.
Let $M = M(\alpha)$ be the maximum integer such that $32^M \alpha <1$.
The fact that $\rho \le \sup |h_D' | \le 4$ implies $M < N$.
For any $r \ge 0$, we denote
$$
J(r) = \{ x \in \mathbb{R} \colon |x - \tilde{x}| < \Dalpha e^{-r} \}.
$$

The next lemma is a similar result of \cite[Lemma~2.6]{Vi97} and its proof
follows closely the ideas and techniques in that paper.

\begin{lemma}
\label{v.lemma2.6}
There exist constants $C_3>0$ and $\beta>0$ such that, given an
admissible curve $\hat{Y}_0 = \graph (Y_0)$ and any
$r \ge \frac{1}{D-1} \left( \frac{1}{D} - \frac{2\eta}{D-1} \right) \log \frac{1}{\alpha}$,
$$
m\left( \{ \theta \in \mathbb{S}^1 \colon \hat{Y}_M(\theta) \in
\mathbb{S}^1 \times J(r-2) \} \right) \le C_3 e^{-5 \beta r}.
$$
\end{lemma}

The proof relies on the following lemma.

\begin{lemma}
\label{v.lemma2.7}
There are $H_1, H_2 \subset \{1, \dots , d \}$ with $\#H_1, \#H_2
\ge d/16]$ such that $|Z_{j_1} - Z_{j_2} | \ge \alpha/100$ for all
$\theta \in \mathbb{S}^1$, $j_1 \in H_1$, and $j_2 \in H_2$.
\end{lemma}

\begin{proof}
See \cite[Lemma~2.7]{Vi97}.
\end{proof}

\begin{proof}[Proof of Lemma~\ref{v.lemma2.6}]
Let $\hat{Y}_j(\theta) = \varphi^j (\theta, Y_0(\theta)) = (g^j(\theta),
Y_j(\theta))$.
We use $C$ to represent any large positive constant depending only on $h_D$.

As $\hat{Y}_0 = \graph (Y_0)$ is an admissible curve, we have
$\osc (Y_0) \le \alpha$ and we claim $\osc(Y_j) \le 4\osc(Y_{j-1}) + 2\alpha$, where
$\osc(Y_j) = \sup Y_j - \inf Y_j$.
Indeed, note that
\begin{align*}
\hat{Y}_j(\theta) & = \varphi^j (\theta, Y_0(\theta)) =
\varphi(\varphi^{j-1} (\theta, Y_0(\theta))) \\
& = \varphi(g^{j-1}(\theta), Y_{j-1}(\theta)) =
(g^j(\theta), f(g^{j-1}(\theta), Y_{j-1}(\theta))).
\end{align*}
Hence, $Y_j(\theta) = f(g^{j-1}(\theta), Y_{j-1}(\theta))$ and so
$$
\osc(Y_j) \le \osc(\alpha \sin 2\pi (g^{j-1}(\theta))) + 4 \osc (Y_{j-1})
+ \alpha \le 4 \osc (Y_{j-1}) + 2\alpha.
$$
Thus, $\osc(Y_j) \le 2\alpha 4^j \le 2(32^{-M} 4^j)$ and
$\displaystyle \osc (Y_{M}) < 2\alpha^{3/5} < \sqrt{\alpha}$.

If $|Y_{M} (\theta) - \tilde{x}| \ge \Dalpha$ for all $\theta$ the lemma
follows since $\{ \theta \in \mathbb{S}^1 \colon \hat{Y}_M(\theta) \in
\mathbb{S}^1 \times J((r-2)(D-1)^2) \} = \emptyset$.
So, let us suppose that for some $\tau \in \mathbb{S}^1$ we have
$|Y_M (\tau) - \tilde{x}| < \Dalpha$ and thus
\begin{equation}
\label{eq.11}
|Y_M (\theta) - \tilde{x}| < 2\Dalpha \quad (< \delta_1) \text{ for every }
\theta \in \mathbb{S}^1.
\end{equation}

Let us denote $\mathcal{O} = \{ h^i(\tilde{x}) \colon i \ge 1 \}$ and
$\delta_j(\theta) = \dist (Y_j(\theta), \mathcal{O})$.
Similar argument of \eqref{eq.5} yields
$$
\delta_{j+i}(\theta) \le C 4^i ( \alpha + |Y_j(\theta) -\tilde{x}|^D).
$$
for all $\theta$, $0 \le j \le M-1$, and $1 \le i \le M-j$.

If $|Y_{j_0}(\tau) - \tilde{x}| \le \Dalpha$ for some $\tau \in
\mathbb{S}^1$ and some $0 \le j_0 \le M-1$ then $\delta_{M}(\tau) \le C
4^{M - j_0} (\alpha + |Y_{j_0}(\tau) -\tilde{x}|^D) < C 4^{M} \alpha <
C\sqrt{\alpha}$, contradicting \eqref{eq.11}.
Hence
\begin{equation}
\label{eq.out}
|Y_{j}(\theta) - \tilde{x}| > \Dalpha \text{ for any }
\theta \in \mathbb{S}^1 \text{ and } 0 \le j \le M-1.
\end{equation}
In addition, note that the above reasoning together with \eqref{eq.11} gives
\begin{equation}
\label{eq.13}
4^{M - j} |Y_{j}(\theta) - \tilde{x}|^D \ge \frac{1}{C} \text{ for all }
\theta \in \mathbb{S}^1 \text{ and } 0 \le j \le M-1,
\end{equation}
here we are taking $C \gg 1/\dist(\tilde{x}, \mathcal{O})$.

To derive uniform bound for the distortion of $\partial_x f$ on iterates of
$\hat{Y}_0$, note that given $0\le j \le M-1$, $(\theta_j, x_j),
(\tau_j, y_j) \in \hat{Y}_j$, and $1\le i \le M-j$ we have
\begin{equation}
\label{eq.bound}
\begin{aligned}
\left| \frac{\partial_x f^i(\theta_j, x_j)}{\partial_x f^i(\tau_j, y_j)}
\right| & = \prod_{m=j}^{j+i-1} \left|\frac{\partial_x f (\theta_m ,x_m)}
{\partial_x f (\tau_m ,y_m)} \right| \\
& =  \prod_{m=j}^{j+i-1} \left| 1 + \frac{\partial_x f (\theta_m ,x_m) -
\partial_x f (\tau_m ,y_m)|} {\partial_{x} f (\tau_m ,y_m)}\right| .
\end{aligned}
\end{equation}
We have
\begin{align*}
|\partial_x f (\theta_m ,x_m) - \partial_x f (\tau_m ,y_m)| & \le 
%\partial_{xx} f (\tilde{\theta} ,\tilde{x}) |x_m -y_m| \\ & \le
D (D-1) A \osc Y_m \le D(D-1)A 2\alpha 4^m.
\end{align*}
On one hand, if $y_m \in \mathcal{M} \setminus I''$ then
$|\partial_x f (\tau_m ,y_m)| \ge 7/4$.
Hence, assuming $\alpha$ small enough,
$$
\frac{|\partial_x f (\theta_m ,x_m) - \partial_x f (\tau_m ,y_m)|}
{|\partial_x f (\tau_m ,y_m)|} \le C \alpha 4^m < \sqrt{\alpha}.
$$
On the other hand, if $y_m \in I''$ using \eqref{eq.13} we get
$$
|\partial_x f (\tau_m ,y_m)| = DA|y_m - \tilde{x}|^{D-1} \ge
\frac{DA}{\left(\frac{1}{C} 4^{M-m}\right)^{(D-1)/D}},
$$
which implies
$$
\frac{|\partial_x f (\theta_m ,x_m) - \partial_x f (\tau_m ,y_m)|}
{|\partial_x f (\tau_m ,y_m)|} \le C \alpha 4^{M} < \sqrt{\alpha}.
$$
Therefore, from \eqref{eq.bound}, using the fact that
$M \approx \log \frac{1}{\tilde{\alpha}}$ and assuming $\alpha$ small enough, we
have
$$
\left| \frac{\partial_x f^i(\theta_j, x_j)}{\partial_x f^i(\tau_j, y_j)}
\right| \le \prod_{m=j}^{j+i-1} (1 + \sqrt{\alpha}) \le
(1+ \sqrt{\alpha})^{2i} \le e^{2M \sqrt{\alpha}} \le 2.
$$
We just proved that given any $0\le j \le M-1$ and $1\le i \le M-j$
$$
\left| \frac{\partial_x f^i(\theta_j, x_j)}{\partial_x f^i(\tau_j, y_j)}
\right| \le 2 \text{ for every } (\theta_j, x_j),(\tau_j, y_j) \in
\hat{Y}_j.
$$

We fix an arbitrary $\hat{y} \in \hat{Y}_0$ and let $\lambda_j =
| \partial_xf^{M-j}(\varphi^j(\hat{y}))|$.
From Lemma~\ref{l.lemma2.5} (also recall \eqref{eq.out} and
\eqref{eq.11}), we have $\lambda_j \ge C_2 \sigma_2^{M-j}$ for
$0 \le j \le M-1$.
On the other hand, the previous inequality gives
\begin{equation}
\label{eq.15}
\frac{1}{2} \frac{\lambda_j}{\lambda_{i+j}} \le
|\partial_x f^i(\theta_j,x_j)| \le 2 \frac{\lambda_j}{\lambda_{i+j}} \quad
\text{ for all } (\theta_j,x_j) \in \hat{Y}_j.
\end{equation}

Now, see \cite[Proof of Proposition 5.2]{BST03}, assume 
$$
\displaystyle r >  \left( \frac{D-1}{D} + 2\eta \right) \log \frac{1}{\alpha}.
$$
Then
$$
 \frac{D-1}{D} \log \frac{1}{\alpha} < r \left(1 + \frac{2D}{D-1}\eta \right)^{-1} < r \left(1 - \frac{\eta}{D-1}\right).
$$
Thus,
$$
\frac{1}{\alpha^{\frac{D-1}{D}}} < e^{r(1-\eta/(D-1))}.
$$
From Corollary~\ref{c.corollary2.3}, for an admissible curve $\hat{X_0} = \graph X_0$, we have, for every $j \ge 1$,
\begin{align*}
m( \{\theta \in \mathbb{S}^1 \colon \hat{X_j}(\theta) \in & J(r-2) \})  \le C_1 \sqrt{\frac{\Dalpha e^{-(r-2)}}{\alpha}} \\
& =  C_1 \sqrt{\frac{e^{-(r-2)}}{\alpha^\frac{D-1}{D}}} \le e^2 C_1 e^{-\frac{\eta}{D-1}r}.
\end{align*}
Take $\beta = \eta/(5(D-1)$ in order to get the result in this case.

Now consider 
$$
\displaystyle \frac{1}{D-1} \left( \frac{1}{D} - \frac{2\eta}{D-1} \right) \log \frac{1}{\alpha} \le r \le \left( \frac{D-1}{D} + 2\eta \right) \log \frac{1}{\alpha}.
$$
It is enough to proof the lemma for $r_0 = \displaystyle \frac{1}{D-1} \left( \frac{1}{D} - \frac{2\eta}{D-1} \right) \log \frac{1}{\alpha}$.
For the other $r'$s the result follows by replacing $\beta$ by $\displaystyle \beta \frac{1/D - 2\eta/(D-1)}{(D-1)^2 (1/D + 2\eta/(D-1))}$.

We fix $K = 400 e^{2(D-1)^2}$ and consider positive integers $t_1<t_2< \dots \le M$
defined by $t_1 = 1$ and
$$
t_{i+1} = \min \{ s \colon t_i < s \le M \text{ and } \lambda_{t_i} \ge
2K \lambda_s\} \quad \text{ (if it exists)}.
$$

We set $ \displaystyle k = k(r_0) = \max \{i \colon \lambda_{t_i} \ge
2K e^{-r_0(D-1)^2} / \alpha^\frac{D-1}{D} \}$.

We claim that there is a constant $\gamma_1 > 0$ such that
$k(r_0) \ge \gamma_1 r_0$.
Indeed, we have $\lambda_{t_i} \le 2K\lambda_{t_{i+1}-1} \le
8K \lambda_{t_{i+1}}$ for all $i$ and so
\begin{equation}
\label{eq.k1}
\lambda_{t_{k+1}} \ge C_2 \sigma_2^{M-1}(8K)^{-k},
\end{equation}
and, by definition,
\begin{equation}
\label{eq.k2}
\lambda_{t_{k+1}} \le 2K e^{-r_0(D-1)^2} / \alpha^{(D-1)/D}.
\end{equation}
From \eqref{eq.k1} and \eqref{eq.k2} we obtain
$$
C_2 \sigma_2^{M-1}(8K)^{-k} \le 2K e^{-r_0(D-1)^2} / \alpha^{(D-1)/D},
$$
which implies
$$
k \log (8K) \ge r_0(D-1)^2 + M \log \sigma_2 - \frac{D-1}{D} \log \frac{1}{\alpha} + C.
$$
It follows from the definition of $\eta$ and $M$ that $\displaystyle M\log \sigma_2 \ge
4 \eta \log \frac{1}{\alpha}$.
Hence
\begin{align*}
k \log (8K) & \ge r_0 (D-1)^2 - \left(\frac{D-1}{D} - 4\eta\right) \log \frac{1}{\alpha} + C \\
& \ge r_0(D-1)^2 - \left( \frac{\frac{D-1}{D} - 4 \eta}{\frac{D-1}{D} - 2\eta} \right) (D-1)^2 r_0 + C \\
& \ge \left(1 - \frac{\frac{D-1}{D} - 4 \eta}{\frac{D-1}{D} - 2\eta} \right)(D-1)^2 r_0 + C \ge 2(D-1)^2 \eta r_0,
\end{align*}
proving the claim for $\gamma_1 = 2(D-1)^2 \eta/ \log (8K)$.

As in the proof of \cite[Lemma~2.6]{Vi97}, for each
$\tilde{l} = (l_1, \dots , l_{M}) \in \{1, \dots, d\}^{M}$ we denote by
$\omega(\tilde{l})$ the only element $\omega \in \mathcal{P}_{M}$ satisfying
$g^i(\omega) \subset [\tilde{\theta}_{l_i-1}, \tilde{\theta}_{l_i}),
i = 1, \dots , M$.
Given $1 \le j \le M$ we let $\hat{Y}_j = \graph (Y_j(\tilde{l})) =
\varphi^j(\hat{Y}_0(\omega(\tilde{l})))$.
We say that $\tilde{l}$ and $\tilde{m}$ are \emph{incompatible} if
$$
|Y_{M}(\tilde{l}, \theta) - Y_{M}(\tilde{m}, \theta)|
\ge 4 e^{(2-r)(D-1)^2} \Dalpha \text{ for all } \theta \in \mathbb{S}^1.
$$
In this case $\hat{Y}_M(\tilde{l})$ and $\hat{Y}_M(\tilde{m})$ can not both
intersect a same vertical segment $\{\theta\} \times J((r-2)(D-1)^2)$.
By Lemma~\ref{v.lemma2.7} there are $H_1', H_1'' \subset \{1, \dots, d\}$
with $\#H_1', \#H_1'' \ge [d/16]$ such that given $l_1' \in H_1'$ and
$l_1'' \in H_1''$ we have
$$
|Y_1(l_1',l_2, \dots , l_{M}, \theta) - Y_1(l_1'', l_2, \dots , l_{M}, \theta)| \ge
\frac{\alpha}{100}
$$
for all $\theta \in \mathbb{S}^1$ and $l_2, \dots , l_{M}$.
Then, by \eqref{eq.15}, the definition of $K$, and the fact that
$1 \le k(r_0)$,
\begin{align*}
|Y_M(l_1', l_2, \dots , l_M, & \theta)  - Y_M(l_1'', l_2 \dots , l_M, \theta)| \ge
\frac{\lambda_1}{2} \frac{\alpha}{100} \ge \\
& \ge \frac{K e^{-r_0(D-1)^2}}{\alpha^{\frac{D-1}{D}}} \frac{\alpha}{100} = 4e^{(2-r_0)(D-1)^2} \Dalpha,
\text{ for all } \theta \in \mathbb{S}^1,
\end{align*}
that is, $(l_1', l_2, \dots , l_{M})$ and $(l_1'', l_2, \dots , l_{M})$ are incompatible
for every $l_2, \dots , l_{M}$.
Furthermore, we claim that all pair of form $(l_1', l_2, \dots l_{t_2-1},
l_{t_2}', \dots , l_{M}')$ and $(l_1'', l_2, \dots l_{t_2-1}, l_{t_2}'',
\dots , l_{M}'')$ are incompatible.
Indeed, note that, as a consequence of \eqref{eq.15} and the definition of
$t_2$, we have
\begin{align*}
|Y_{t_2}(l_1', l_2 \dots, l_{M}, \theta) - Y_{t_2}(l_1'', l_2 \dots, l_{M}, \theta)|
\ge \frac{\lambda_1}{2\lambda_{t_2}} \frac{\alpha}{100}
\ge K \frac{\alpha}{100} = 4e^{2(D-1)^2} \alpha,
\end{align*}
for all $\theta \in \mathbb{S}^1$.

On the other hand,
\begin{align*}
|Y_{t_2}(l_1', l_2, \dots, & l_{t_2-1}, l_{t_2}', \dots , l_{M}')(\theta) -
Y_{t_2}(l_1', l_2, \dots, l_{t_2-1}, l_{t_2}, \dots, l_{M}) (\theta)| \le \\
& \le \osc (\varphi(\hat{Y}_{t_2-1}(l_1', l_2, \dots, l_{t_2-1})))
\le 8\alpha,
\end{align*}
for every $\theta \in \mathbb{S}^1$, and similarly for
$Y_{t_2}(l_1'', \dots )$.
Thus, as $\lambda_{t_2} \le k(r_0)$, we have
\begin{align*}
|Y_{M}(l_1', l_2, & \dots l_{t_2-1}, l_{t_2}', \dots, l_{M}', \theta) -
Y_{M}(l_1'', l_2, \dots l_{t_2-1}, l_{t_2}'', \dots , l_{M}'', \theta)| \ge \\
& \ge \frac{\lambda_{t_2}}{2}(4e^{2(D-1)^2} - 16) \alpha \ge \frac{K e^{-r_0(D-1)^2}}{\alpha^{\frac{D-1}{D}}} (4e^{2(D-1)^2} - 16) \alpha = \\
& = 400 e^{(2-r_0)(D-1)^2} (4e^{2(D-1)} - 16) \Dalpha \ge 4 e^{(2-r_0)(D-1)^2} \Dalpha,
\end{align*}
for all $\theta \in \mathbb{S}^1$, proving the claim.
Proceeding as above for each $t_i$, we get that each segment $\{ \theta \} \times J((r-2)(D-1)^2)$ intersects at most $d^{M - k(r_0)} \cdot (d - [d/16])^{k(r_0)}$ admissible curves $\hat{Y}_M(\tilde{l})$, see the proof of \cite[Lemma~2.6]{Vi97} for details.
The fact that $\displaystyle M \le \text{const } \log \frac{1}{\alpha}$ implies that 
$$
|(g^M)'| \ge (d- \alpha)^M \ge \text{const} d^M.
$$
Therefore,
\begin{align*}
m \left( \{ \theta  \colon \hat{Y}_M ( \theta) \in \mathbb{S}^1 \times J((r-2)(D-1)^2) \} \right) & \le \frac{d^M((d-[d/16])/d)^{k(r_0) }}{(d - \alpha)^M} \\
& \le \text{const} \left( \frac{99}{100} \right)^{\gamma_1 r_0}.
\end{align*}
The lemma follows by taking $\displaystyle \beta = \frac{\gamma_1}{5} \log \left( \frac{100}{99} \right)$.
\end{proof}

\section{Proof of Theorem~\ref{mt.general}}

To conclude the proof we continue following closely the proof of \cite{Vi97}.
We provide the details when the critical point is of order $D \ge 2$ (note that taking $D=2$, the proof is the same).
See \cite{Vi97}, for more details.

Let us recall some ingredients and notations. 
For $n \ge 1$ be sufficiently large and fixed, we define $m \ge 1$ by $m^2 \le n < (m+1)^2$ and take $l = m - M$, where
$M = M(\alpha)$ is as above.
Note that $l \approx m \approx \sqrt{n}$ as long as
$n \gg \log \frac{1}{\alpha}$.
Let $\hat{X}_0$ be an arbitrary admissible curve.
Given $1 \le \nu \le n$ and $\omega_{\nu + l} \in \mathcal{P}_{\nu + l}$,
we set $\gamma = \varphi^\nu(\hat{X}_0|\omega_{\nu + l})$.
We say that $\nu$ is 
\begin{itemize}
\item a I$_n$-\textit{situation} for $\theta \in \omega_{\nu + l}$ if $\gamma \cap (\mathbb{S}^1 \times J(0)) \neq \emptyset$ but $\gamma \cap (\mathbb{S}^1 \times J(m)) = \emptyset$;
\item  a II$_n$-\textit{situation} for $\theta \in \omega_{\nu + l}$ if $\gamma \cap (\mathbb{S}^1 \times J(m)) \neq \emptyset$.
\end{itemize}

This setting is exactly the same of \cite{Vi97}, let us remind some consequences for completeness.
It follows from Lemma~\ref{v.lemma2.1} that $\gamma$ is the graph of a function defined on $g^\nu (\omega_{\nu + l}) \in \mathcal{P}_l$ and whose derivative is bounded above by $\alpha$. 
So, the diameter in the $x$-direction is bounded by $\alpha(d-\alpha)^{-l} \ll \sqrt[D]{\alpha} \; e^{-m}$.
This means that whenever $\nu$ is a II$_n$-situation for $\omega_{\nu + l}$ then $\gamma \subset (\mathbb{S}^1 \times J(m-1))$.

Let $B_2(n) = \{ \theta \in \mathbb{S}^1 \colon \text{ some } 1 \le \nu \le n \text{ is a II}_n\text{-situation for } \theta \}$.  
From Corollary~\ref{c.corollary2.3}, we obtain
\begin{equation}
\label{eq.b2}
m(B_2(n)) \le n C_1 \sqrt{\frac{|J(m-1)|}{\alpha}} \le \text{const}\; \alpha^{-(D-1)/(2D)}ne^{-m/2} \le \text{const}\; e^{-\sqrt{n}/4}.
\end{equation}

Thus, from this point we focus on values $\theta$ having no II$_n$-situations in $[1,n]$.
Let $1 \le \nu_1 < \dots < \nu_s \le n$ be the I$_n$-situation of $\theta$.
The definition of $N$ implies $\nu_{i+1} \ge \nu_i + N$ for every $i$; in particular $(s-1)N \le n$.

For each $\nu = \nu_i$ we fix $r = r_i \in \{1, \dots, m\}$ minimum such
that $\gamma \cup (\mathbb{S}^1 \times J(r)) = \emptyset$.
Then, by Lemma~\ref{v.lemma2.4} and definition of $J(r)$ 
\begin{align*}
\prod_{\nu_i}^{\nu_i + N -1} \left|\partial_x f(\hat{X}_j(\theta))\right|
& \ge (\Dalpha e^{-r_i})^{D-1} \alpha^{-1 + \frac{\eta}{D-1}} \\ 
& = e^{-(D-1)r_i} \alpha^{\frac{D-1}{D} -1  + \frac{\eta}{D-1}} \\
& = e^{-(D-1)r_i} \alpha^{-\frac{1}{D} + \frac{\eta}{D-1}},
\end{align*}
for each $1 \le i < s$.

Lemma~\ref{l.lemma2.5} gives
$$
\prod_{1}^{\nu_1 - 1} \left|\partial_x f(\hat{X}_j(\theta))\right| \ge C_2 \sigma_2^{\nu_1-1} \quad \text{ and } \quad
\prod_{\nu_i + N}^{\nu_{i+1} - 1} \left|\partial_x f(\hat{X}_j(\theta))\right| \ge C_2 \sigma_2^{\nu_{i+1} - \nu_i - N},
$$
for every $1 \le i < s$ and also
\begin{align*}
\prod_{\nu_s}^{n} \left|\partial_x f(\hat{X}_j(\theta))\right| & \ge (DA - \alpha) |x_{\nu_s} - \tilde{x}|^{D-1} C_2 (\Dalpha^{D-1} \sigma_2^{n - \nu_s} \\
& \ge \text{const} \; (\Dalpha)^{2(D-1)} e^{-(D-1)r_s} \sigma_2^{n - \nu_s}.  
\end{align*}
Altogether this yields the following lower bound for
$\log \prod_{j=1}^n \left| \partial_x f(\hat{X}_j(\theta)) \right|$:
\begin{align*}
(n - (s-1) N) \log \sigma_2 + \sum_{i-1}^s & \left( \left( \frac{1}{D}- \frac{\eta}{D-1}
\right) \log \frac{1}{\alpha} - (D-1)r_i \right) \\
& -s\;\text{const} - \frac{2D-1}{D} \log \frac{1}{\alpha}.
\end{align*}

We consider $G = \left\{ i \colon r_i \ge \frac{1}{D-1}\left(\frac{1}{D} - \frac{2\eta}{D-1}
\right) \log \frac{1}{\alpha} \right\}$ (note that $G$ depends on $\theta$) and then
\begin{align*}
\sum_{i-1}^s \left( \left( \frac{1}{D}- \frac{\eta}{D-1}  \right)  \log \frac{1}{\alpha} - (D-1)r_i \right)  & \ge - \sum_{i \in G} (D-1) r_i + \eta s \log \frac{1}{\alpha} \\
& \ge - (D-1)  \sum_{i \in G} r_i + \gamma_2Ns,
\end{align*}
for some $\gamma_2 > 0$ independent of $\alpha$ or $n$ (because $N \approx \text{ const } \log (1/\alpha))$.
Thus, we have
\begin{align*}
\log \prod_{j=1}^n \left| \partial_x f(\hat{X}_j(\theta)) \right| & \ge (D+1) cn - (D-1) \sum_{i \in G}  r_i -s\;\text{const} - \frac{2D-1}{D} \log \frac{1}{\alpha} \\
& \ge D cn - (D-1) \sum_{i \in G}  r_i,
\end{align*}
where $c = \frac{1}{D+1} \min\{\gamma_2, \log \sigma_2 \}$ and we use $n \gg \log \frac{1}{\alpha} \approx N \gg 1$.

Now we introduce 
$$
B_1(n) = \{ \theta \in \mathbb{S}^1 \colon \sum_{i \in G} r_i \ge cn \}
$$
and set $E_n = B_1(n) \cup B_2(n)$.
Then 
$$
\log \prod_{j=1}^n \left| \partial_x f(\hat{X}_j(\theta)) \right| \ge cn \text{ for every } \theta \in \mathbb{S}^1 \setminus E_n.
$$

Considering \eqref{eq.b2}, it remains to show that $m(B_1(n)) \le \text{const } e^{-\gamma \sqrt{n}}$, for some $\gamma > 0$. 
This is done using Lemma~\ref{v.lemma2.6} and a large deviations argument. 
For completeness, we provide a sketch of the proof, for details see \cite[Section 2]{Vi97}.
Let $0 \le q \le m-1$ be fixed and denote $G_q = \{ i \in G \colon \nu_i \equiv q \mod m \}$.
Let $m_q  =  \max \{ j \colon mj + q \le n \}$ (note $m_q \approx m \approx \sqrt{n}$) and for each $0 \le j \le m_q$ we let $\hat{r}_j = r_i$ if $mj = q = \nu_i$, for some $i \in G_q$, and $\hat{r}_j = 0$ otherwise.
Notice that $G_q$ and the $\hat{r}_j$ are functions of $\theta$. 
So, we write
$$
\Omega_q (\rho_0 , \cdots , \rho_{m_q}) = \{ \theta \in \mathbb{S}^1 \setminus B_2(n) \colon \hat{r}_j = \rho_j \text{ for } 0 \le j \le m_j \}
$$
where for each $j$ either $\rho_j = 0$ or $\rho_j \ge \displaystyle \frac{1}{D-1} \left(\frac{1}{D} - \frac{2\eta}{D-1} \right) \log \frac{1}{\alpha}$; we also assume the $\rho_j$ not to be simultaneously zero.
For $0 \le j \le m_q$ and $\omega_{mj + q + l} \in \mathcal{P}_{mj + q + l}$, $\hat{Y}_0 = \varphi^{mj + q + l} (\hat{X}_0|\omega{mj + q = l})$ is an admissible curve and we have defined $l$ in such a way that $mj + q + l  = m(j+1) + 1 - M$. 
Therefore, we can apply Lemma~\ref{v.lemma2.6} to obtain
$$
m(\{ \theta \in \omega_{mj + q + l} \colon \hat{r}_{j+1} = \rho \}) \le C_*C_3 e^{-5\beta \rho},
$$
for all $\displaystyle \rho \ge  \frac{1}{D-1} \left(\frac{1}{D} - \frac{2\eta}{D-1}  \right) \log \frac{1}{\alpha}$.
Here $C_*$ is a uniform upper bound for the metric distortion of the iterates of $g$. 
Thus, we are in position to apply the large deviation argument as in \cite{Vi97} to finish the proof.

%\bibliographystyle{plain}
%\bibliography{bib}

\bigskip

\flushleft

{\bf Vanderlei Horita} (vanderlei.horita\@@unesp.br)\\
Departamento de Matem\'{a}tica, IBILCE/UNESP \\
Rua Crist\'{o}v\~{a}o Colombo 2265\\
15054-000 S. J. Rio Preto, SP, Brazil

\bigskip

\flushleft

{\bf Nivaldo Muniz} (nmuniz\@@demat.ufma.br)\\
Departamento de Matem\'{a}tica, UFMA \\
Avenida dos Portugueses, S/N \\
65000-000 S\~{a}o Lu\'{\i}s, MA, Brazil

\bigskip

\flushleft

{\bf Olivier Sester} (olivier.sester@univ-eiffel.fr) \\
Univ Gustave Eiffel, Univ Paris Est, CNRS, LAMA UMR8050 \\
F-77447 Marne-la-Vall\'ee, France

\end{document}